\documentclass[11pt,letterpaper]{article}
\usepackage[margin=1in]{caption}
\usepackage{geometry}
\usepackage{amsmath,amssymb,color, mathtools, tikz, tikz-cd, multicol,subfig}
\usepackage{amscd}
\usepackage{amsthm}
\numberwithin{equation}{section}
\def\proof{\smallskip\noindent {\it Proof: \ }}
\def\endproof{\hfill$\square$\medskip}
\newtheorem{theorem}{Theorem}[section]
\newtheorem{proposition}[theorem]{Proposition}

\newtheorem{question}[theorem]{Question}
\newtheorem{corollary}[theorem]{Corollary}

\newtheorem{lemma}[theorem]{Lemma}

\newtheorem{problem}[theorem]{Problem}

\theoremstyle{definition}
\newtheorem{definition}[theorem]{Definition}

\newtheorem{remark}[theorem]{Remark}

\newcommand{\conv}{\mathrm{conv}}
\newcommand{\lk}{\mathrm{lk}}
\newcommand{\st}{\mathrm{st}}
\newcommand{\aff}{\mathrm{aff}}

\newcommand{\R}{\mathbb{R}}
\newcommand{\F}{\mathcal{F}}

\newcommand{\T}{\mathcal{T}}

\title{Transversal numbers of simplicial polytopes, spheres, and pure complexes.}
\author{
	Isabella Novik\thanks{Research of IN is partially\textsl{} supported by NSF grant  DMS-2246399.}\\
	\small Department of Mathematics\\[-0.8ex]
	\small University of Washington\\[-0.8ex]
	\small Seattle, WA 98195-4350, USA\\[-0.8ex]
	\small \texttt{novik@uw.edu}
	\and 
	Hailun Zheng\thanks{Research of HZ is partially\textsl{} supported by NSF grant DMS-2246793.} \\
	\small Department of Mathematics\\[-0.8ex]
	\small University of Hawai`i at M\={a}noa\\[-0.8ex]
	\small 2565 McCarthy Mall, Honolulu, HI 96822, USA \\[-0.8ex]
	\small \texttt{hailunz@hawaii.edu}
}
\begin{document}
\maketitle
		\begin{abstract}
		We prove new upper and lower bounds on transversal numbers of several classes of simplicial complexes. Specifically, we establish an upper bound on the transversal numbers of pure simplicial complexes in terms of the number of vertices and the number of facets, and then provide constructions of pure simplicial complexes whose transversal numbers come close to this bound.
		We introduce a new family of $d$-dimensional polytopes that could be considered as ``siblings'' of cyclic polytopes and show that the transversal ratios of such odd-dimensional polytopes are $2/5-o(1)$. The previous record for the transversal ratios of $(2k+1)$-polytopes was $1/(k+1)$. Finally, we construct infinite families of $3$-, $4$-, and $5$-dimensional simplicial spheres with transversal ratios converging to $4/7$, $1/2$, and $6/11$, respectively. The previous record was $11/21$, $2/5$, and $1/2$, respectively.
		\end{abstract}

\section{Introduction} The goal of this paper is to establish several upper and lower bounds on the transversal numbers and transversal ratios of pure simplicial complexes, simplicial polytopes, and simplicial spheres.	

A transversal of a hypergraph $H=(V,E)$ with vertex set $V$ and edge set $E$ is defined as a subset of $V$ that intersects {\em all} edges of $H$. The transversal number of $H$, which we denote by $T(H)$, is the minimum cardinality of a transversal of $H$, and the transversal ratio of $H$, $\tau(H)$, is $T(H)/|V|$. 

The class of $d$-uniform hypergraphs is closely related to the class of pure $(d-1)$-dimensional simplicial complexes:  if $\Delta$  is a pure simplicial $(d-1)$-dimensional complex with vertex set $V$, then the hypergraph  $H(\Delta)=(V,\F)$ whose set of edges $\F$ is the set of facets of $\Delta$, is a $d$-uniform hypergraph. Conversely, every $d$-uniform hypergraph $H=(V,E)$ determines a pure $(d-1)$-dimensional simplicial complex on vertex set $V$ whose set of facets is given by $E$. For a pure simplicial complex $\Delta$, we define 
the transversal number of $\Delta$, $T(\Delta)$, and the transversal ratio of $\Delta$, $\tau(\Delta)$, as $T(H(\Delta))$ and $\tau(H(\Delta))$, respectively. Throughout this paper, we will use the language of pure simplicial complexes.

The starting point of our paper is the following problem raised by Tur\'an:

\begin{problem} {\rm\cite{Turan}}
	Determine $f(n, m, d)=\max T(H)$, where $H$ ranges over all $d$-uniform hypergraphs with $n$ vertices and with $m$ edges (equivalently, where $H$ ranges over all pure $(d-1)$-dimensional simplicial complexes with $n$ vertices and $m$ facets). 
\end{problem}

A lot of work has been done on Tur\'an's problem in the case where $m$ is at most linear in $n$, especially in relation to Tuza's problem; see \cite{Alon90} and \cite{ChvatalMcDiarmid}. Here, motivated by results obtained and questions raised in \cite{Alonetall02,BriggsDobbinsLee,Heise-etal,LeeNevo}, we are interested in transversal numbers of simplicial complexes with interesting geometric and topological properties, such as, for instance, pseudomanifolds, Eulerian complexes, and simplicial spheres. (We will discuss definitions of these objects in Section 2. For now, we merely mention that all simplicial spheres are Eulerian complexes, 
 and all pseudomanifolds are pure complexes.) According to the Lower Bound Theorem \cite{Barnette-LBT-pseudomanifolds, Barnette73} and the Upper Bound Theorem \cite{McMullen70, Stanley75}, a simplicial sphere of dimension $d-1\geq 2$ with $n$ vertices has at least $(d-1)n-(d^2-d-2)$ and at most  $\binom{n-\lfloor \frac{d+1}{2}\rfloor}{n-d}+\binom{n-\lfloor \frac{d+2}{2}\rfloor}{n-d}$ facets. For this reason, we are primarily interested in the $m\gg n$ case of Tur\'an's problem, and, especially, in the case where $m$ is about $n^{\lfloor d/2\rfloor}$.

Not much appears to be known about this case of Tur\'an's problem. Among our results in this direction are:
\begin{enumerate}
\item If $d\geq 2$ and $\Delta$ is a pure $(d-1)$-dimensional simplicial complex with $n$ vertices and $m$ facets, then for $n$ sufficiently large, $T(\Delta)\leq n(1-\frac{1}{e}m^{-1/d})+1$; see Theorem \ref{lm: pure complex transversal}. In particular, there is an absolute constant $c>0$ independent of $n$ and $d$ such that for every Eulerian complex $\Delta$  with $n$ vertices, $T(\Delta)\leq n-cn^{1/2}$ if $d$ is even, and $T(\Delta)\leq n-cn^{(d+1)/(2d)}$ if $d$ is odd; see Corollary \ref{cor:Eulerian}.

\item Conversely,  for every $d\geq 2$ and $n\ll m\ll n^{(d+1)/2}$,  there exists a family of pure $(d-1)$-dimensional complexes $\{\Delta^{n,m}_d\}$ such that $\Delta^{n,m}_d$ has $n$ vertices, $C_d m$ facets, and  $T(\Delta^{n,m}_d)\geq n\left(1-C_d'n^{1/(d-1)}m^{-1/(d-1)}\right)$, where $C_d, C'_d>0$ are absolute constants independent of $n$ and $m$; see Proposition \ref{prop: pure}.
\end{enumerate}

In the second part of the paper, our focus is specifically on 
 transversal numbers of simplicial spheres and (boundary complexes of) simplicial polytopes. It follows easily from the Four Color Theorem that the transversal number of any simplicial $2$-sphere with $n$ vertices is at most $n/2$ (see \cite[Prop.~3.1]{BriggsDobbinsLee}); furthermore, for some values of $n$, $n\to\infty$, there exist simplicial $2$-spheres with $n$ vertices  whose transversal numbers are equal to $n/2$; see \cite[Section 4]{BriggsDobbinsLee}. This motivates studying asymptotics of transversal ratios of simplicial polytopes and simplicial spheres --- a problem that was raised, for instance, in \cite{Alonetall02,BriggsDobbinsLee}. Specifically, for $d\geq 2$, we let
		\begin{eqnarray*}\tau^P_d &:=&\limsup_{n\to\infty}\max\{\tau(\Delta): \Delta=\partial Q, \;\text{where $Q$ is a simplicial $d$-polytope with $n$ vertices}\},\\
			\tau^S_d&:=&\limsup_{n\to \infty}\max\{\tau(\Delta): \Delta\; \text{is a simplicial $(d-1)$-sphere with $n$ vertices}\}.
		\end{eqnarray*}
	
	Very little is known about these two sequences at present. Clearly, $\tau^P_2=\tau^S_2=1/2$ and $0<\tau^P_d\leq \tau^S_d\leq 1$ for all $d\geq 3$, and, as was mentioned above, $\tau^P_3=\tau^S_3=1/2$. Using cyclic polytopes, Briggs et al. (see \cite[Prop.~3.6]{BriggsDobbinsLee}) showed that for all $k\geq 2$,  $\tau^P_{2k}\geq 1/2$; additionally, with computer help, they proved that $\tau^S_4\geq 11/21$ (see \cite[Theorem 1.2]{BriggsDobbinsLee}). On the other hand, for $d=2k+1\geq 5$, they were only able to prove the bound $\tau^P_{2k+1}\geq 1/(k+1)$. (In this regard, it should be pointed out that all odd-dimensional cyclic polytopes have transversal number {\em two} independently of their number of vertices!) Finally, using certain families of (highly neighborly) $2k$-spheres and  centrally symmetric $2k$-spheres, the authors showed in \cite{NZ-transversal} that $\tau^S_{2k+1}\geq 2/5$. This completes the list of bounds that have been known so far.

The main contribution of this paper are the following new lower bounds:
	\begin{enumerate}
	\item $\tau^P_{2k+1}\geq 2/5$ for all $k\geq 2$; see Theorem \ref{thm:poly-2/5}.
	\item $\tau^S_{4}\geq 4/7$, $\tau^S_{5}\geq 1/2$, and $\tau^S_{6}\geq 6/11$; see Theorems \ref{thm:3-sphere-5/8}, \ref{thm:5-sphere-6/11}, and \ref{thm:1/2-in-dim4}.
	\end{enumerate}

Most of our proofs either use or are inspired by the properties of the cyclic polytopes. To construct families of pure complexes with relatively large transversal numbers (see Section 3), we use a variation of Gale's evenness condition \cite[Example 0.6]{Ziegler}. To prove that $\tau^P_{2k+1}\geq 2/5$ for all $k\geq 2$, we utilize Shemer's sewing technique \cite{Shemer}. 
Using this technique, for all $d\geq 4$, we introduce a new family of simplicial $d$-polytopes that could be considered as siblings of cyclic polytopes (see Section 4). The idea behind constructing simplicial (in fact, PL) spheres of dimension $3\leq d-1\leq 5$ with $n$ vertices that have higher transversal ratios than previously known resembles that of \cite[Section 5]{BriggsDobbinsLee}: specifically, we start with the boundary complex of the cyclic $d$-polytope with $n$ vertices (if $d$ is even) or with the boundary complex of the sibling of the cyclic $d$-polytope  introduced in Section 4 (if $d$ is odd), and then apply a sequence of strategically chosen bistellar flips (or local retriangulations) to increase the transversal ratio; see Section 5 for details. 

The structure of the paper is as follows. In Section 2, we discuss basics of hypergraphs and simplicial complexes, with a particular emphasis on simplicial polytopes and spheres; we also review there some results and definitions pertaining to transversal numbers.  Section 3 is devoted to transversal numbers of pure simplicial complexes. In Section 4, we construct siblings of cyclic polytopes and use them to prove that  $\tau^P_{2k+1}\geq 2/5$ for all $k\geq 2$. Finally, Section 5 presents constructions of PL spheres of dimensions $3$, $4$, and $5$ with higher transversal numbers than previously known: Section 5.1 discusses a possible general approach, while Sections 5.2 and 5.3 provide specific details in dimensions $3$ and $5$, and dimension $4$, respectively. (We note that Sections 4 and 5 could be read independently of Section 3.) We close in Section 6 with a few open problems.

We hope that the proof techniques introduced in the paper and, especially, the siblings of the cyclic polytopes will be of interest in their own right.

\section{Preliminaries}
\subsection{Hypergraphs and simplicial complexes}
We start with basic definitions and results pertaining to hypergraphs and simplicial complexes.
A \emph{hypergraph} $H=(V,E)$ consists of a (finite) set $V$, called the {\em vertex set} of $H$, and a collection $E$ of subsets of $V$, called the {\em edge set} of $H$. We say that $H$ is {\em $r$-uniform} if each edge of $H$ has size $r$. In particular, graphs are $2$-uniform hypergraphs. We usually assume that every vertex belongs to some edge. 

Similarly, a {\em simplicial complex} $\Delta$ with vertex set $V=V(\Delta)$ is a collection of subsets of $V$ that is closed under inclusion, that is, if $F\in\Delta$ and $G\subset F$, then $G\in\Delta$. The elements of $\Delta$ are called {\em faces}. We usually assume that every $v\in V$ forms a face which for brevity we denote by $v$ instead of $\{v\}$. The {\em dimension} of a face $F\in\Delta$ is $|F|-1$, and the {\em dimension} of $\Delta$ is the maximum dimension of its faces. The number of $i$-dimensional faces of $\Delta$ is denoted by $f_i(\Delta)$. 
 
The maximal under inclusion faces of a simplicial complex are called {\em facets}. A simplicial complex $\Delta$ is {\em pure} if all facets of $\Delta$ have the same dimension; in this case, the faces of dimension $\dim\Delta-1$ are called {\em ridges}.  Two important examples of pure simplicial complexes with vertex set $V$ are the {\em simplex}, $\overline{V}$, consisting of all subsets of $V$, and the {\em boundary complex} of $\overline{V}$, $\partial\overline{V}$, consisting of all subsets of $V$ but $V$ itself. When $V = \{v\}$ is a singleton, we write $\overline{v}$ instead of $\overline{\{v\}}$.

We now list several operations that allow us to construct new simplicial complexes from the old ones. If $\Lambda\subset\Delta$ are pure simplicial complexes of the same dimension, then we write $\Delta\backslash\Lambda$ to denote the pure simplicial complex generated by the facets of $\Delta$ that are not facets of $\Lambda$. If $\Delta$ and $\Gamma$ are two simplicial complexes on disjoint vertex sets then their {\em join}, denoted $\Delta*\Gamma$, is
\[\Delta*\Gamma=\{F\cup G : F\in\Delta, \ G\in\Gamma\}.\] The join of $\Delta$ with a $0$-dimensional simplex $\overline{v}$ is called the {\em cone} over $\Delta$, and is denoted $v*\Delta$. For a simplicial complex $\Delta$ and a face $F\in\Delta$, the {\em star} and the {\em link} of $F$ in $\Delta$ are the following subcomplexes of $\Delta$:
\[\st(F,\Delta)=\{G\in\Delta : F\cup G\in\Delta\} \quad \mbox{and} \quad \lk(F,\Delta)=\{G\in\st(F,\Delta) : F\cap G=\emptyset\}.\]
(If $\Delta$ is fixed or understood from context, we sometimes write $\st(F)$ and $\lk(F)$ instead of $\st(F,\Delta)$ and $\lk(F,\Delta)$.)
Thus, if $v$ is a vertex of $\Delta$, then $\st(v,\Delta)=v*\lk(v,\Delta)$. Finally, for a simplicial complex
$\Delta$ on $V$ and a subset $W$ of $V$, we denote by $\Delta[W] = \{F \in \Delta : F \subseteq W\}$ the {\em restriction} of $\Delta$ to $W$. A subcomplex $\Gamma$ of $\Delta$ is {\em induced} if it is of the form $\Delta[W]$ for some $W\subseteq V$.

\subsection{Simplicial polytopes}
We now discuss some basics of polytopes. An excellent reference to this material is \cite{Ziegler}.

 A {\em polytope} $P$ is the convex hull of finitely many points in a Euclidean space. The {\em dimension of $P$} is defined as the dimension of the affine hull of $P$. For brevity, we refer to a $d$-dimensional polytope as a {\em $d$-polytope}. An example of a $d$-polytope is a {\em (geometric) simplex} defined as the convex hull of $d+1$ affinely independent points in $\R^d$.

Let $P\subset \R^d$ be a $d$-polytope. A {\em supporting hyperplane} $L$ of $P$ is any hyperplane in $\R^d$ such that all points of $P$ lie on the same side of $L$. A {\em (proper) face} of $P$ is the intersection of $P$ with a supporting hyperplane. (This includes the empty face.) A face of a polytope is by itself a polytope. A face $F$ of $P$ is called an {\em $i$-face} if $\dim F=i$; $0$-faces are called {\em vertices} and $(d-1)$-faces are called {\em facets}. A polytope is {\em simplicial} if all of its facets are simplices. The {\em boundary complex} of a simplicial $d$-polytope $P$, $\partial P$, consists of the vertex sets of all proper faces of $P$; this is a simplicial complex of dimension $d-1$. For instance, if $\sigma$ is a (geometric) $d$-simplex with vertex set $V$, then $\partial \sigma=\partial\overline{V}$.

A family of simplicial polytopes that plays a crucial role in this paper is that of {\em cyclic polytopes}. Let $d\geq 2$ and let $M: \R\to \R^d, t\mapsto (t, t^2, \dots, t^d)$ be the ($d$-th) \emph{moment curve}. Given any $n\geq d+1$ distinct numbers $t_1 < t_2 < \dots < t_n$ in $\R$, we define the \emph{cyclic polytope}, $C(n, d)$, as the convex hull of the points $M(t_1), \dots, M(t_n)$. The polytope $C(n,d)$ has several remarkable properties (see \cite[Example 0.6]{Ziegler}), including that it is a simplicial
$d$-polytope with $n$ vertices whose combinatorial type is {\em independent} of the choice of $t_1, t_2, \dots, t_n$. For this reason, when talking about $\partial C(n,d)$, we label the vertices by elements of $[n]=[1,n]:=\{1,2,\dots,n\}$, with $i$ serving as a label for $M(t_i)$. Specifically, the set of facets of $C(n, d)$ is completely characterized by the following result, known as the Gale evenness condition:
\begin{theorem}\label{thm: Gale}
	A $d$-subset $T$ of $[n]$ forms a facet of $\partial C(n, d)$ if and only if any two elements of $[n]\backslash T$ are separated by an even number of elements from $T$. In particular, if $d=2k$, then every $d$-set of the form $\{i_1<i_1+1<i_2<i_2+1<\cdots<i_k<i_k+1\} \subseteq [n]$ is a facet of $\partial C(n, d)$.
\end{theorem}

One immediate corollary is that $\partial C(n, d)$ is {\em $\lfloor d/2\rfloor$-neighborly}, that is, every $\lfloor d/2\rfloor$ vertices of $[n]$ form the vertex set of a face of $\partial C(n, d)$.

\subsection{Simplicial spheres}
Via the notion of a {\em geometric realization}, one associates with a simplicial complex $\Delta$ a topological space, denoted $\|\Delta\|$: this space is built out of geometric simplices in a way that every two simplices intersect along a common (possibly empty) face and the collection of vertex sets of faces of $\|\Delta\|$ is $\Delta$. We often say that $\Delta$ has certain
geometric or topological properties if $\|\Delta\|$ does. For instance, we say that $\Delta$ is a {\em simplicial $(d-1)$-sphere} if $\|\Delta\|$ is homeomorphic to a $(d-1)$-dimensional sphere, and that $\Delta$ is a {\em simplicial $d$-ball} if $\|\Delta\|$  is homeomorphic to a $d$-dimensional ball.   As an example, $\overline{V}$ is a simplicial ball and $\partial\overline{V}$ is a simplicial sphere.

Assume $\Delta$ is a simplicial ball. Then every ridge of $\Delta$ is contained in at most two facets. The boundary complex of $\Delta$, $\partial\Delta$, is the pure $(d-1)$-dimensional complex generated by those ridges that are contained in a unique facet.
The faces of $\partial \Delta$ are called {\em boundary faces} of $\Delta$; all other faces of $\Delta$ are called {\em interior faces}.  For instance, what we previously denoted by $\partial\overline{V}$ is indeed the boundary complex of $\overline{V}$; here $V$ is the only interior face of $\overline{V}$. As another example, if $\Delta$ and $\Gamma$ are simplicial balls with disjoint vertex sets, then $\Delta*\Gamma$ is also a simplicial ball and $\partial(\Delta*\Gamma)=(\partial\Delta*\Gamma)\cup(\Delta*\partial\Gamma)$.

The boundary complex of any simplicial  polytope is a simplicial sphere. In light of this, a simplicial $(d-1)$-sphere is called {\em polytopal} if it can be realized as the boundary complex of a simplicial $d$-polytope. It follows from Steinitz' theorem that all simplicial $2$-spheres are polytopal; however, for $d\geq 4$, most of simplicial $(d-1)$-spheres are not polytopal; see \cite{GoodmanPollack,Kal,NeSanWil,PfeiZieg}.

An important subclass of simplicial spheres and balls is the class of PL spheres and PL balls. We say that $\Delta$ is a {\em PL $d$-ball} if $\Delta$ is PL homeomorphic to a $d$-simplex, while $\Delta$ is a {\em PL $(d-1)$-sphere} if it is PL homeomorphic to the boundary complex of a $d$-simplex. For example, the boundary complex of a PL $d$-ball is a PL $(d-1)$-sphere, and so is the boundary complex of a simplicial $d$-polytope. 
 If $P$ is a simplicial polytope and $\Delta$ is a PL ball such that $\partial P=\partial\Delta$, then $\Delta$ is called a {\em triangulation} of $P$.

Let $\Delta$ be a PL $(d-1)$-sphere. If $\Delta$ contains an induced subcomplex $\overline{A}*\partial \overline{B}$, where $A$ and $B$ are disjoint nonempty subsets of $V(\Delta)$ with $|A|+|B|=d+1$, then we can perform a {\em bistellar flip} on $\Delta$ by replacing $\overline{A}*\partial \overline{B}$ with $\partial \overline{A}*\overline{B}$. The resulting complex is again a PL $(d-1)$-sphere. The following result due to Pachner \cite{Pachner} allows to easily search through the space of PL spheres.

\begin{theorem} {\rm\cite{Pachner}}
	A simplicial complex $\Delta$ is a PL $(d-1)$-sphere if and only if $\Delta$  can be obtained from the boundary complex of a $d$-simplex by a finite sequence of bistellar flips.
\end{theorem}
It is worth noting that for $d\leq 4$, every simplicial $(d-1)$-sphere is PL; however, for $d\geq 6$ there exist simplicial $(d-1)$-spheres that are not PL. (To the best of our knowledge, the question of whether there exist simplicial $4$-spheres that are not PL remains a major open problem.) The reader is referred to \cite{Hudson,Lickorish} for additional background on PL topology.

\subsection{Transversal numbers and asymptotic notations}
Let $H=(V,E)$ be a hypergraph. A {\em transversal} of $H$ is a subset $\T$ of $V$ that intersects with all edges of $H$. The {\em transversal number} of $H$, $T(H)$, is the minimum size of a transversal of $H$, and the {\em transversal ratio} of $H$ is $\tau(H):=T(H)/|V|$. 

Note that a pure $(d-1)$-dimensional simplicial complex $\Delta$ with vertex set $V$ is uniquely determined by the $d$-uniform hypergraph $H(\Delta)=(V, \F)$ where $\F$ is the set of facets of $\Delta$. We define the {\em transversal number} and the {\em transversal ratio} of $\Delta$, $T(\Delta)$ and $\tau(\Delta)$, respectively, as the transversal number and the transversal ratio of the hypergraph $H(\Delta)$.  

Transversal numbers of hypergraphs (or of associated pure complexes) are closely related to other well-studied invariants of hypergraphs. A {\em (weak) independent set} of a hypergraph $H=(V,E)$ is a subset $I$ of $V$ that contains no edge of $H$. The {\em independence number} of $H$, $\alpha(H)$, is the maximum size of an independent set.  A {\em weak coloring} $\kappa$ of $H$ is an assignment of colors to the vertices of $H$ so that no edge is monochromatic. Equivalently, the pre-image of any color is an independent set. We say that $\kappa$ is a {\em strong coloring} 
if the restriction of $\kappa$ to any edge of $H$ is an injective function. 
The {\em weak chromatic number} of $H$, $\chi_w(H)$,  is  the minimum number of colors in a weak coloring of $H$; the {\em strong chromatic number} of $H$, $\chi_s(H)$, is defined analogously.  

The following lemma is an immediate consequence of these definitions.

\begin{lemma}\label{lm: transversal property}
	Let $r\geq 2$. Let $H$ be an $r$-uniform hypergraph with $n$ vertices. Then 
	\begin{enumerate}
		\item $T(H) = n-\alpha(H)$.
		\item $T(H) \leq  n-n/\chi_w(H)$ and  $T(H)\leq  \frac{(\chi_s(H)-(r-1))n}{\chi_s(H)}$.
	\end{enumerate}
\end{lemma}
The classical theorem of Tur\'an asserts that a graph $G$ with $n$ vertices and the average vertex degree $d_G$ has an independent set of size at least $n/(d_G+1)$; see \cite[Chapter 41]{AignerZiegler}. This theorem together with part 1 of the above lemma then implies

\begin{corollary}\label{lm: Turan}
	Let $d\geq 4$  and let $\Delta$ be a pure $(d-1)$-dimensional complex with $f_0(\Delta)=n$ and $f_1(\Delta)=\ell$. Then $T(\Delta)\leq n-\frac{n}{2\ell/n+1}$.
\end{corollary}

Recall that the Four Color Theorem states that if $\Delta$ is a simplicial $2$-sphere, then $\chi_s(H(\Delta))$ is at most $4$. Together with part 2 of Lemma \ref{lm: transversal property}, this yields
\begin{corollary}
	The transversal number of a simplicial $2$-sphere $\Delta$ is at most $\frac{1}{2}f_0(\Delta)$.
\end{corollary}

In this paper, we discuss the {\em asymptotics} of transversal numbers and transversal ratios of various classes of pure simplicial complexes. Below we give a quick review of asymptotic notations. Let $f, g$ be two functions from the set of nonnegative integers to itself. We say that $g(n)=o(f(n))$ if $\lim_{n\to\infty} g(n)/f(n)=0$. Throughout, we use $f(n)\gg g(n)$ and $g(n)=o(f(n))$ interchangeably. We write $g(n)=O(f(n))$ if there is a positive constant $C$ such that $g(n)\leq Cf(n)$ for all (sufficiently large) $n$; similarly, we write $g(n)=\Omega(f(n))$ if there is a positive constant $c$ such that $g(n)\geq cf(n)$ for all $n$. Finally, we write $g(n)=\Theta(f(n))$ if $g(n)=O(f(n))$ and $g(n)=\Omega(f(n))$. Thus, $g(n)= n-\Omega(f(n))$ means that there is a constant $c>0$ such that $g(n)\leq n-cf(n)$ for all $n$,  while $g(n)= n-\Theta(f(n))$ means that there are positive constants $c$ and $C$ such that $n-Cf(n)\leq g(n)\leq n-cf(n)$ for all $n$.
	
\section{Transversal numbers of pure complexes}
In this section we establish an upper bound on the transversal numbers of pure simplicial complexes in terms of the number of vertices and the number of facets; see Theorem \ref{lm: pure complex transversal}. We then discuss constructions of pure complexes whose transversal numbers come close to this bound. Along the way we briefly touch on Eulerian complexes and pseudomanifolds.

\begin{theorem} \label{lm: pure complex transversal}
		Let $d\geq 2$ and let $\Delta$ be a pure $(d-1)$-dimensional complex with $n$ vertices and $m$ facets. Then for $n$ sufficiently large, $T(\Delta) \leq n+1-\frac{1}{e}nm^{-1/d}$.
	\end{theorem}
	\begin{proof}
		Let $V_0=\emptyset$ and let $\Delta_0=\Delta$. For each $0\leq i<n$, we inductively define $V_{i+1}:=V_i\cup v_{i+1}$ and $\Delta_{i+1}:=\Delta[V(\Delta)\backslash V_{i+1}]$ by choosing $v_{i+1}\in V(\Delta)\backslash V_i$
in a way that maximizes $f_{d-1}(\Delta_i)-f_{d-1}(\Delta_{i+1})$.
		Since $$\sum_{u\in V(\Delta_i)} f_{d-1}(\st(u, \Delta_i))= df_{d-1}(\Delta_i),$$ it follows that $f_{d-1}(\st(v_{i+1}, \Delta_i))\geq \frac{df_{d-1}(\Delta_i)}{n-|V_i|}$, and hence that $$f_{d-1}(\Delta_{i+1})= f_{d-1}(\Delta_i)-f_{d-1}(\st(v_{i+1}, \Delta_i))\leq f_{d-1}(\Delta_i)\cdot \left(1-\frac{d}{n-i}\right).$$
		We then conclude by induction on $i$ that for $i>d$ and $n>i+d$,
		\begin{equation} \label{eq:facets}
			\begin{split}
				f_{d-1}(\Delta_{i+1})&\leq m\frac{(n-d)(n-d-1)\dots (n-i-d)}{n(n-1)\dots(n-i)}\\
				& = m \frac{(n-i-1)(n-i-2)\dots(n-i-d)}{n(n-1)\dots(n-d+1)}=m \frac{\binom{n-i-1}{d}}{\binom{n}{d}}.
			\end{split}
		\end{equation}
		
		Since for all positive $a$ and $b$, $(a/b)^b\leq \binom{a}{b}\leq (ea/b)^b$, eq.~(\refeq{eq:facets}) can be rewritten as
		\[
			\begin{split}
				f_{d-1}(\Delta_{i+1})&\leq m \frac{(e(n-i-1)/d)^{d}}{(n/d)^{d}}\\
				& = me^d\left(\frac{n-i-1}{n}\right)^d.
			\end{split}
		\]
It remains to note that $V_{i+1}$ is a transversal of $\Delta$ if and only if $f_{d-1}(\Delta_{i+1})=0$. Consequently, $T(\Delta) \leq i+1$ where $i$ is the smallest integer such that
		$$\left(\frac{n-i-1}{n}\right)^d < \frac{ 1}{me^d}, \; \text{or, equivalently}, \, n-i-1 < \frac{1}{e}\cdot nm^{-1/d}.$$
	 The desired bound $T(\Delta) \leq n+1-\frac{1}{e}nm^{-1/d}$ follows. 
	\end{proof}
	
	We recall Klee's Upper Bound Theorem; see \cite{Klee-UBT}. A simplicial complex $\Delta$ of dimension $d-1$ is \emph{Eulerian} if $\tilde{\chi}(\lk(F, \Delta))=(-1)^{d-|F|-1}$ for all $F\in \Delta$ (including the empty face); here $\tilde{\chi}$ denotes the {\em reduced Euler characteristic}. For instance, all simplicial spheres are Eulerian.
	\begin{theorem}
		Let $d\geq 2$ and let $\Delta$ be a $(d-1)$-dimensional Eulerian complex with $n$ vertices. Then for $n$ sufficiently large, $f_i(\Delta)\leq f_i(C(n,d))$ for all $i$. In particular, $f_{d-1}(\Delta)  \leq 2n^{\lfloor d/2\rfloor}$.
	\end{theorem}
	
	\begin{corollary}  \label{cor:Eulerian}
		Let $d\geq 2$. Then for every $(d-1)$-dimensional Eulerian complex $\Delta$ with $n$ vertices,  $T(\Delta) \leq n-c n^{1/2}$ if $d$ is even and $T(\Delta) \leq n-c n^{(d+1)/(2d)}$ if $d$ is odd. Here $c>0$ is an absolute constant independent of $n$ and $d$.
	\end{corollary}

	\begin{remark}
	In \cite[Theorems 1.3 and 6.7]{BriggsDobbinsLee}, it is proved that the weak chromatic number of a simplicial $d$-polytope or a simplicial $(d-1)$-sphere with $n$ vertices satisfies $\chi_w=O\big(n^{\frac{\lceil (d-1)/2\rceil-1}{d-1}}\big)$. Hence the transversal number of a simplicial $(d-1)$-sphere with $n$ vertices is $\leq n-n/\chi_w=n-\Omega\big(n^{\frac{\lfloor (d+1)/2\rfloor}{d-1}}\big)$, regardless of the number of facets. This bound is slightly better than the one in Corollary \ref{cor:Eulerian}. On the other hand, the bound of Corollary~\ref{cor:Eulerian} 
	holds for all Eulerian complexes rather than just simplicial spheres.
\end{remark}

    Our next goal is to construct pure complexes with transversal numbers close to the bound in Theorem~\ref{lm: pure complex transversal}; see Corollary \ref{cor:pure}, Proposition \ref{prop: pure}, and Corollary \ref{cor: pseudomanifold}. These constructions are inspired by the Gale evenness condition.
	\begin{definition}
		Let $k\geq 1$ be a constant. Consider any function $s: \mathbb{N}\to \mathbb{N}$, $n\mapsto s(n)$, such that $s=o(n)$. For $n$ sufficiently large,  define $\mathcal{F}^{n, d}_{s}$ to be the following family of subsets of $[n]$:
		\begin{enumerate}
			\item If $d=2k$, then let
			$$\mathcal{F}^{n,d}_{s}=\{\{i_1<i_2<\dots< i_{2k-1}< i_{2k}\}: \,\exists 1\leq \ell \leq s\;\; \text{s.~t.} \; \forall\, 1\leq j\leq k, \; i_{2j}-i_{2j-1}=\ell \}.$$
			\item If $d=2k+1$, then let
			$$\mathcal{F}^{n, d}_{s}=\left\{\{i_1<i_2<\dots<i_{2k}< i_{2k+1}\} : 
			\begin{array}{ll} i_{2k+1}-i_{2k}\leq s\mbox{ and }\\ 
			 \exists 1\leq \ell \leq s \;\mbox{ s.~t. }  \forall 1\leq j\leq k, \; i_{2j}-i_{2j-1}=\ell \end{array}\right\}.$$
		\end{enumerate}
	\end{definition}
	\noindent It follows that $|\mathcal{F}^{n,d}_{s}|= \binom{n-k}{k}+\binom{n-2k}{k}+\dots+\binom{n-sk}{k}=\Theta(sn^k)$ if $d=2k$ and, similarly, $|\mathcal{F}^{n,d}_{s}|=\Theta(s^2n^k)$ if $d=2k+1$.
	\begin{theorem} \label{prop:T(F_s)}
	Let $d\geq 2$. The transversal number of the pure $(d-1)$-dimensional complex with the set of facets given by $\mathcal{F}^{n,d}_{s}$ is at least $n(1-\frac{1}{s+1})-ks$ if $d=2k$, and at least $n(1-\frac{2}{s+2})-2ks$ if $d=2k+1$. Furthermore, $\mathcal{F}^{n, d}_{s}$ has a transversal of size $n-\lfloor \frac{n}{s+1}\rfloor$.
	\end{theorem}
	
	\begin{proof}
 	Independently of whether $d$ is even or odd, the complement of multiples of $s+1$ in $[n]$ is a transversal of $\mathcal{F}^{n, d}_{s}$ . Hence $T(\mathcal{F}^{n, d}_{s})\leq n-\lfloor \frac{n}{s+1}\rfloor$.

		For $d=2k$, consider a transversal $\mathcal{T}$ of $\mathcal{F}^{n,d}_s$ and let $\T^c=[n]\backslash \T=\{a_1<a_2<\dots<a_p\}$. Note that $$(a_2-a_1)+(a_3-a_2)+\dots+(a_p-a_{p-1})=a_p-a_1<n,$$
		and so at least one of the following inequalities holds: $$(a_2-a_1)+(a_4-a_3)+\dots <n/2, \quad 
		(a_3-a_2)+(a_5-a_4)+\dots <n/2.$$
		Assume w.l.o.g.~that $\sum_{k=1}^{\lfloor p/2\rfloor} (a_{2k}-a_{2k-1})<n/2$. Consider the multiset $M=\{a_2-a_1,a_4-a_3, \dots\}$. If some number $1\leq \ell\leq s$ appears in $M$ at least $k$ times, say, $$\ell=a_{2j_1}-a_{2j_1-1}= \cdots= a_{2j_k}-a_{2j_k-1},$$ then $\{a_{2j_1-1}, a_{2j_1},\ldots, a_{2j_k-1}, a_{2j_k}\}$ is an element of $\mathcal{F}^{n, d}_{s}$ that is disjoint from $\T$. Hence for $\T$ to be a transversal, $M$ must contain at most $k-1$ elements equal to $t$ for each $1\leq t \leq s$; all remaining elements of $M$ must be $\geq s+1$. Since the total number of elements in $M$ is $\lfloor  p/2\rfloor\geq\frac{n-|\T|-1}{2}$, it follows that 
		\begin{eqnarray*}\frac{n}{2}>\sum_{k=1}^{\lfloor p/2\rfloor} (a_{2k}-a_{2k-1})&\geq& (k-1)(1+2+\dots +s)+(|M|-(k-1)s)(s+1)\\
		&\geq&\frac{(n-|\T|-1-(k-1)s)(s+1)}{2}.
		\end{eqnarray*}
		In other words, $\frac{n}{s+1}\geq n-|\T|-1-(k-1)s$, and hence $|\T|\geq n(1-\frac{1}{s+1})-ks$.
		
		Now consider the case of $d=2k+1$. Again let $\mathcal{T}$ be a transversal of $\mathcal{F}^{n,d}_s$, let $\mathcal{T}^c=\{a_1<a_2<\dots<a_p\}$, and consider the multiset $M=\{a_2-a_1, \dots, a_p-a_{p-1}\}$. W.l.o.g. assume that $p$ is odd. We call a pair $\{a_{2i}-a_{2i-1}, a_{2i+1}-a_{2i}\}$ large if $a_{2i+1}-a_{2i-1}\geq s+2$, and call it small otherwise. Since the sum of all elements of $M$ is less than $n$, the number of large pairs is at most $\frac{n}{s+2}$. If there are at least $(k-1)s+1$ small pairs with $x_{2i+1}-x_{2i-1}\leq s+1$ (in particular, $x_{2i}-x_{2i-1}\leq s$ and $x_{2i+1}-x_{2i}\leq s$), then there exists an $1\leq \ell \leq s$ and indices $i_1<i_2<\dots<i_k$ such that $x_{2i_1}-x_{2i_1-1}=\dots=x_{2i_k}-x_{2i_k-1}=\ell$ and $x_{2i_k+1}-x_{2i_k}\leq s$. But then $\{x_{2i_1-1}, x_{2i_1}, \dots, x_{2i_k-1}, x_{2i_k}, x_{2i_k+1}\}$ is an element of $\mathcal{F}^{n,d}_s$ that is disjoint from $\mathcal{T}$, contradicting our assumption that $\mathcal{T}$ is a transversal of $\mathcal{F}^{n, d}_s$.
		
		Hence the number of small pairs is at most $(k-1)s$ and $\frac{|\mathcal{T}^c|}{2}\leq \left( (k-1)s+\frac{n}{s+2}\right) +1$. Consequently,
		$$|\mathcal{T}|=n-|\mathcal{T}^c|\geq n-2(k-1)s-\frac{2n}{s+2}-2\geq n\left(1-\frac{2}{s+2}\right)-2ks.$$ 
		This completes the proof.
	\end{proof}
	
	The following result is an immediate consequence of Theorem \ref{prop:T(F_s)} applied to the pure complex whose facets are given by $\mathcal{F}^{n,d}_{s}$ with $s\leq\sqrt{n}$. When $s=\sqrt{n}$, the transversal number of the resulting complex almost matches the bound of Theorem \ref{lm: pure complex transversal}.

	\begin{corollary} \label{cor:pure} 
 Fix $d\geq 2$, and let $s=s(n)$ be such that $1\leq s \leq \sqrt{n}$. Then for all sufficiently large $n$,
		there exists a pure $(d-1)$-dimensional complex with $n$ vertices and $\Theta(n^{\lfloor d/2\rfloor}s^{1+(d \mod 2)})$ facets whose transversal number is $n-\Theta(n/s)$. In particular, there exists a pure $(d-1)$-dimensional complex with $n$ vertices and $\Theta(n^{\frac{d+1}{2}})$ facets whose transversal number is $n-\Theta(n^{1/2})$.
	\end{corollary}

Next we discuss two variations of this construction. The first variation complements Corollary \ref{cor:pure} and allows us to construct a pure $(d-1)$-dimensional complex with $n$ vertices and $ m \ll n^{\frac{d+1}{2}}$ facets that has a relatively large transversal number. It is based on the following simple idea. If $\Delta$ is a pure $(d-1)$-dimensional complex with $n$ vertices, then the disjoint union of $b$ copies of $\Delta$, denoted by $b\Delta$, has $f_{0}(b\Delta)=bf_{0}(\Delta)$, $f_{d-1}(b\Delta)=bf_{d-1}(\Delta)$, and $T(b\Delta)=bT(\Delta)$.

Given $n,m$ and $d$, we choose $b$ as some function $b(n,m,d)$, to be specified below, and we let $s=\sqrt{n/b}$. Consider the  $(d-1)$-dimensional complex $\Delta^{n/b, d}_s$ whose set of facets is given by $\mathcal{F}^{n/b, d}_s$. Then $b\Delta^{n/b, d}_s$ has $n$ vertices and $\Theta \big(b \cdot (n/b)^{(d+1)/2}\big)$ facets; furthermore, $T(b\Delta^{n/b,d}_s)=b(n/b)\big(1-\Theta(\frac{1}{s})\big)=n-\Theta(n/s)$. Taking $b=\big(\frac{n^{d+1}}{m^2}\big)^{\frac{1}{d-1}}$ ensures that the resulting complex has $\Theta(m)$ facets and implies the following result.

\begin{proposition}\label{prop: pure}
Fix $d\geq 2$. 
Then for $n$ and $m$ such that $1 \ll n \ll m \ll n^{\frac{d+1}{2}}$, there exists a pure $(d-1)$-dimensional complex with $n$ vertices and $\Theta(m)$ facets whose transversal number is $n-\Theta\left(n^{\frac{d}{d-1}}m^{-\frac{1}{d-1}}\right)$. In particular, for all sufficiently large $n$, there exists a pure $(d-1)$-dimensional complex with $n$ vertices and $\Theta\big(n^{\lfloor d/2\rfloor}\big)$ facets whose transversal number is $n-\Theta\big(n^{\lceil d/2 \rceil/(d-1)}\big)$.
\end{proposition}

Our second variation produces a family of odd-dimensional pseudomanifolds with large transversal numbers. Recall that a pure complex is a \emph{pseudomanifold} if every ridge is contained in at most two facets.
 
\begin{definition}
	Fix $k\geq 2$ and let $n$ be sufficiently large. Let $a, s$ be functions of $n$ such that $a=o(n)$, $s=o(\frac{n}{a})$, and $a, s \gg k$. Partition $n$ into $a$ intervals of size as equal as possible and denote these intervals by $I_1, I_2, \dots, I_a$. Let $\mathcal{H}^n_{a, s}$ be the following collection of sets
	$$\left\{ \{i_1, i_1+\ell, \dots, i_k, i_k+\ell\}: \begin{array}{ll} 1\leq \ell \leq s \mbox{ and }\\ \exists 1\leq b_1<\dots< b_k\leq a \; \mbox{ s.t. } \{i_j, i_j+\ell\}\subset I_{b_j}\, \forall 1\leq j\leq k \end{array} \right\}.$$
\end{definition}

\begin{lemma}
	The pure complex with the set of facets given by $\mathcal{H}^n_{a, s}$ is a $(2k-1)$-dimensional pseudomanifold with $n$ vertices and $m=\Theta(sn^k)$ facets.
\end{lemma}
\begin{proof}
	Let $G$ be a ridge. It must be of the form $\{i, i_2, i_2+\ell, \dots, i_k, i_k+\ell\}$, where for each $2\leq j\leq k$, $\{i_j, i_j+\ell\}$ is a subset of one of the intervals. Assume w.l.o.g. that $i\in I_1$. If $i\leq \ell$, then $G$ is contained in a unique facet $G\cup \{i+\ell\}$. If $\max I_1 -\ell <i\leq \max I_1$, then $G$ is contained in a unique facet $G\cup \{i-\ell\}$. Otherwise, $G$ is contained in exactly two facets $G\cup \{i-\ell\}$ and $G\cup \{i+\ell\}$. Hence the complex is a $(2k-1)$-dimensional pseudomanifold. The number of facets is $\Theta\left(s\binom{a}{k}(n/a)^k\right)=\Theta(sn^k)$.
\end{proof}
\begin{proposition}
	Consider $a, s$ such that $a = \Omega(s\ln s)$ and $s=o(n/a)$. Then the transversal number of the complex with the set of facets given by $\mathcal{H}^n_{a, s}$  is $n-\Theta(n/s)$. 
\end{proposition}
\begin{proof}
The complement of multiples of $s+1$ in $[n]$ is a transversal; hence the transversal number is $n-\Omega(n/s)$. To see that the transversal number is $ n-O(n/s)$, let
 $\T$ be a transversal and let $\T_j=\T\cap I_j$. We let $\T_j^c=([n]\backslash \T)\cap I_j$ and write it as $\T^c_j=\{x_1^j<x_2^j<\dots <x^j_{p_j}\}$. Consider the multiset $M_j=\{x_2^j-x_1^j, \dots, x^j_{p_j}-x^j_{p_j-1}\}$ and let $m_j=\min M_j$.
	
	First we claim that for every $1\leq t\leq s$, at most $k-1$ of $m_j$'s are equal to $t$. Indeed, if for some $t$ there are at least $k$ of them, then there exist $1\leq j_1<j_2<\dots<j_k\leq a$ and $x^{j_\ell}_{u_\ell}, x^{j_\ell}_{u_\ell-1}\in \T^c_{j_\ell}$ such that $$t=x^{j_1}_{u_1}-x^{j_1}_{u_1-1}=\dots=x^{j_k}_{u_k}-x^{j_k}_{u_k-1}.$$ 
	But then $F=\{x^{j_1}_{u_1-1}, x^{j_1}_{u_1},\dots,x^{j_k}_{u_k-1}, x^{j_k}_{u_k}\}$ is an element of $\mathcal{H}^n_{a, s}$ disjoint from $\T$, contradicting our assumption that $\T$ is a transversal of the complex.
	
	Hence for all but $\leq (k-1)s$ intervals $I_j$, $m_j\geq s+1$. As the sum of all elements of $M_j$ is at most $n/a$, it follows that for each of these intervals, $|M_j|\leq \frac{n}{a(s+1)}$. Consequently, $|\T^c_j| \leq  \frac{n}{a(s+1)} +1$ and hence $|\T_j|\geq n/a-|\mathcal{T}^c_j|\geq \frac{n}{a}-\frac{n}{a(s+1)}-1$.
	
	Similarly, for each of at most $k-1$ intervals $I_j$ with $m_j=t\leq s$, the size of $\T_j$ is  $\geq \frac{n}{a}(1-\frac{1}{t})-1$. Thus,
	\begin{equation*}
		\begin{split}
			|\T|&=\sum_{j=1}^a|\T_j|\\
			&\geq \left(a-(k-1)s\right)\left(\frac{n}{a}-\frac{n}{a(s+1)}-1\right)+(k-1)\sum_{t=1}^{s}\left(\frac{n}{a}\big(1-\frac{1}{t}\big)-1\right)\\
			&= n\left(1-\frac{1}{s+1}\right)-\frac{(k-1)n}{a}\left( \frac{1}{1}+\frac{1}{2}+\dots+\frac{1}{s}-\frac{s}{s+1}\right)-a\\
			&\geq n\left(1-\frac{1}{s+1}-\frac{k-1}{a}\ln s\right)-a.
		\end{split}
	\end{equation*}
	Since $a= \Omega\big(s\ln s\big)$ and $s=o(n/a)$, it follows that $\frac{k-1}{a}\ln s=O\big(\frac{1}{s}\big)$ and $a=o(n/s)$. The above inequality then completes the proof.
\end{proof}

Letting $a=\sqrt{n}$ and $s=\frac{\sqrt{n}}{\ln n}$, we obtain
\begin{corollary}\label{cor: pseudomanifold}
	Let $k\geq 2$ be a constant. Then for all sufficiently large $n$, there exists a $(2k-1)$-dimensional pseudomanifold $\Pi^n$ 
	with $n$ vertices, $m=\Theta(\frac{n^{k+1/2}}{\ln n})$ facets, and $T(\Pi^n)=n-\Theta(\sqrt{n}\ln n)$.
\end{corollary}

\section{Transversal ratios of polytopes}
The goal of this section is to introduce a new family of simplicial polytopes and use it to show that for all $k\geq 2$, $\tau^P_{2k+1}\geq 2/5$. These polytopes could be considered as ``siblings'' of cyclic polytopes. Their construction relies on {\em sewing} --- a powerful tool introduced by Shemer \cite{Shemer} that amounts to inductively constructing polytopes by adding one vertex at a time. Thus, our first task is to define this operation.

Given a simplicial $d$-polytope $P$ in $\R^d$ and a facet $F$ of $P$, let $H_F=\aff(F)$ be the supporting hyperplane that defines $F$ and let $H^-_F$  ($H^+_F$, resp.) be the open half-space determined by $H_F$ that contains the interior of $P$ (is disjoint from $P$, resp.).  We say that a point $p\in\R^d\backslash P$ lies \emph{beneath} $F$ if $p\in H_F^-$ and that $p$ lies \emph{beyond} $F$ if  $p\in H^+_F$. We also say that $p$ lies \emph{exactly beyond} a set of facets $\mathcal{F}$ if $p$ lies beyond every facet in $\mathcal{F}$ and  beneath all other facets of $P$.

Given a flag of faces $F_1\subsetneq F_2 \subsetneq\dots \subsetneq F_\ell$ in $P$, let $$\Gamma=  \st(F_1)\backslash \left(\st(F_2)\backslash \big( \dots\backslash \big(\st(F_{\ell-1})\backslash \st(F_\ell)\big)\dots\big)\right)$$ be a subcomplex of $\partial P$, where the stars are computed in $\partial P$. (For instance, if $\ell=3$, then $\Gamma=\st(F_1)\backslash \big(\st(F_2)\backslash \st(F_3)\big)$ while if $\ell=4$, then $\Gamma=\st(F_1)\backslash\big(\st(F_2)\backslash\big(\st(F_3)\backslash \st(F_4)\big)\big)$.) It is shown in \cite[Lemma 4.4]{Shemer} that there exists a point $p$ that lies exactly beyond the facets of $\Gamma$; hence we can sew this point $p$ onto $P$ to construct a simplicial polytope $\conv(P\cup p)$ whose boundary complex is obtained from $\partial P$ by replacing $\Gamma$ with $\partial \Gamma*p$. 

To introduce our new family of simplicial polytopes, we need to review some properties of the cyclic polytopes. This requires the following definition of simplicial complexes that were considered, for instance, in \cite{BilleraLee}, \cite{Kal}, and \cite{NZ-neighborly}.
 \begin{definition} \label{def:B}
	Let $k\geq 0$ and $1\leq a<b\leq n$. Define $B([a,b],2k-1)$ to be the pure simplicial complex of dimension $2k-1$ generated by the following facets
	$$\{\{i_1, i_1+1,i_2, i_2+1,\dots, i_k, i_k+1\}: a\leq i_1, i_1+1< i_2, \dots, i_{k-1}+1< i_{k}\leq b-1\},$$ 
	and define $B([a, b], 2k):=B([a, b-1], 2k-1)*b.$ (Thus, $B([a,b],-1)=\{\emptyset\}$ and $B([a,b],0)=\overline{b}$.)
\end{definition}

The following lemma summarizes several properties of the cyclic polytopes. All parts of the lemma follow easily from the Gale evenness condition (see Theorem \ref{thm: Gale}) and Definition \ref{def:B}; see also the proof of \cite[Theorem 2.4]{NZ-neighborly}. 
Recall that if $B$ is a PL $d$-ball, then an interior face of $B$ is any face that is not a face of $\partial B$. A  PL $d$-ball $B$ is called {\em $i$-stacked} (for some $0 \leq i \leq  d$) if all interior faces of $B$ are of dimension $\geq d-i$.  For instance, any simplex is $0$-stacked; $1$-stacked balls are also known in the literature as stacked balls.

\begin{lemma}\label{lm: cyclic polytope properties}
	Let $d\geq 1$ and $n>d+1$. Then 
	\begin{enumerate}
		\item $B([1,n], d)$ is a subcomplex of $\partial C(n, d+1)$.
		In fact, 
		\begin{eqnarray*}
		\partial C(n,2k) &=& B([1,n], 2k-1) \cup \left( \overline{\{1,n\}}*B([2,n-1], 2k-3)\right), \mbox{ and }\\
		\partial C(n,2k-1) &=& \left( 1*B([2,n], 2k-3)\right)  \cup \left( B([1,n-1], 2k-3)*n\right).
	\end{eqnarray*}
	\item Let $F_0=\emptyset, F_1=\{n\}, F_2=\{n-1, n\}, \dots, F_{2k+1}=\{n-2k, n-2k+1, \dots, n\}$. 
Then for $d=2k-1$, $$B([1, n], 2k-1)=\st(F_0)\backslash \left(\st(F_1)\backslash \big( \dots\backslash \big(\st(F_{2k-1})\backslash \st(F_{2k})\big)\dots\big)\right),$$ where the stars are computed in $\partial C(n, 2k)$, while for $d=2k$, $$B([1, n], 2k)=\st(F_1)\backslash  \left(\st(F_2)\backslash \big( \dots\backslash \big(\st(F_{2k})\backslash \st(F_{2k+1})\big)\dots\big)\right),$$ where the stars are computed in $\partial C(n,2k+1)$. In particular, $B([1, n], d)$ is a PL $d$-ball for all $d$.
		
	\item $B([1, n], d)$ is $\lceil d/2\rceil$-neighborly $\lceil d/2\rceil$-stacked.
	
	\item  $\partial B([1,n],d)=\lk(n+1, \partial C(n+1,d+1))=\partial C(n,d)$. In other words, $B([1, n], d)$ is a triangulation of $C(n,d)$. Furthermore, when $d\geq 2$, $C(n+1,d)$ is obtained from $C(n,d)$ by sewing a new vertex $n+1$ onto $C(n,d)$ so that it is placed exactly beyond the facets of  $\overline{\{1,n\}}*B([2,n-1], 2k-3)$ if $d=2k$ and exactly beyond the facets of $B([1,n-1], 2k-3)*n$ if $d=2k-1$.
	\end{enumerate}
\end{lemma}

We are now in a position to discuss a generalization of $B([1, n], 2k-1)$. By an interval in $[n]$ of size $i$ we mean a subset of $[n]$ consisting of $i$ consecutive integers. For $J=(j_1, \dots, j_m)$, let $\|J\|:=\sum_{k=1}^m j_k$.
\begin{definition}
	Let $J=(j_1, \dots, j_m)$ where each $j_i\geq 2$ and $\|J\|=d+1$. For $n>d$, define $\Gamma^J_n$ as the $d$-dimensional complex generated by all facets of the form $I=I_1\cup I_2\cup \dots \cup I_m$, where $I_1,\ldots,I_m$ are pairwise disjoint intervals in $[n]$ of sizes $j_1,\ldots,j_m$, respectively, and each $I_i$ lies to the left of $I_{i+1}$. 
\end{definition}

\noindent For instance, if $J$ is a $k$-tuple $(2,2,\ldots,2)$, then $\Gamma^J_n$ is  $B([1, n], 2k-1)$. We will see that several properties of $B([1,n],2k-1)$, such as neighborliness and stackedness, continue to hold in the generality of $\Gamma^J_n$. The proof will rely on the following standard lemma (see \cite[Lemma 2.2]{NZ-cs-neighborly}).
\begin{lemma}\label{lm: PL union}
	Let $B_1$ and $B_2$ be $m$-stacked PL $d$-balls. If $B_1\cap B_2 \subseteq \partial B_1\cap \partial B_2  \subsetneq \partial B_1$ is an $(m-1)$-stacked PL $(d-1)$-ball, then $B_1\cup B_2$ is also an $m$-stacked PL $d$-ball. Further, if $B$ is a $p$-stacked PL ball with $V(B_1)\cap V(B)=\emptyset$, then $B_1*B$ is an $(m+p)$-stacked PL ball.
\end{lemma}

\begin{lemma} \label{lem:m-neighb-m-stacked}
	Let $J=(j_1, \dots, j_m)$ with $\|J\|=d+1$ and all $j_i\geq 2$. Then for all $n > d$, $\Gamma^J_n$ is an $m$-neighborly $m$-stacked PL $d$-ball.
\end{lemma}
\begin{proof}
	By definition, $\Gamma^J_n$ is $m$-neighborly. We prove by induction on both $m$ and the number of vertices $n$ that $\Gamma^J_n$ is an $m$-stacked PL $d$-ball.
	
	First, when $m=1$, the collection of facets is given by $\big\{\{i, i+1, \dots, i+d-1, i+d\} :1\leq i\leq n-d\big\}$, and hence $\Gamma^J_n$ is a stacked ball. Similarly, when $n=d+1$, $\Gamma^J_n$ is the $d$-simplex, and hence the claim also holds.
	
	Let $\hat{J}=(j_1, j_2, \dots, j_{m-1})$. Then $\Gamma^J_{n+1}=\Gamma^J_n\cup \left( \overline{\{n-j_m+2, \dots, n+1\}}*\Gamma^{\hat{J}}_{n-j_m+1}\right) $. Note that by induction on $n$, $\Gamma^J_n$ is an $m$-stacked PL $d$-ball, and by induction on $m$, $\overline{\{n-j_m+2, \dots, n+1\}}*\Gamma^{\hat{J}}_{n-j_m+1}$ is an $(m-1)$-stacked PL $d$-ball. Furthermore,  by definitions (and since $n+1$ is not a vertex of $\Gamma^J_n$),
	$$\Gamma^J_n\cap \left( \overline{\{n-j_m+2, \dots, n+1\}}*\Gamma^{\hat{J}}_{n-j_m+1}\right)=\overline{\{n-j_m+2, \dots, n\}}*\Gamma^{\hat{J}}_{n-j_m+1}.$$
	 By Lemma \ref{lm: PL union}, this join is an $(m-1)$-stacked PL $(d-1)$-ball. Since it is contained in $\partial\Gamma^J_{n}$ and also in $\partial \big(\overline{\{n-j_m+2, \dots, n+1\}}*\Gamma^{\hat{J}}_{n-j_m+1}\big)$, another application of Lemma \ref{lm: PL union} shows that $\Gamma^J_{n+1}$ is an $m$-stacked PL $d$-ball.
\end{proof}

We are ready to define the promised ``sibling'' of the cyclic polytope. We start by defining it as a simplicial sphere, and then show in Lemma \ref{lem:polytopality} that it is indeed the boundary of a polytope.

\begin{definition}
Let $d\geq 4$. Let $J=(2,2,\dots,2,3)$ when $d$ is even and $J=(2,2,\dots,2,4)$ when $d$ is odd, where in both cases $\|J\|=d+1$. For all $n> d$, define $D(n,d-1)$ to be the boundary complex of the PL $d$-ball $\Gamma^J_n$. In particular, $D(n,d-1)$ is a PL $(d-1)$-sphere.
\end{definition}

According to the above definition, the facets of $D(n,d-1)$ are those ridges of $\Gamma^J_n$ that are contained in a unique facet of $\Gamma^J_n$. Using the definition of $\Gamma^J_n$ then easily implies the following lemma, which provides the complete set of facets of $D(n, d-1)$ (cf.~the Gale evenness condition).

\begin{lemma} \label{lem:facets-of-D}
	If $k\geq 2$ and $n\geq 2k+1$, then the set of facets of $D(n, 2k-1)$ is given by  
	\begin{enumerate}
		\item $\tau\cup \{ i_k, i_k+2\}$, where $\tau\in B([1,i_k-1], 2k-3)$ and $i_k\leq n-2$,
		\item $\tau\cup \{n-1,n\}$, where $\tau\in B([1,n-2], 2k-3)$,
		\item $\tau \cup \{1, i_k, i_k+1, i_k+2\}$, where $\tau\in B([2, i_k-1], 2k-5)$ and $i_k\leq n-2$.
	\end{enumerate}
	Similarly, if $k\geq 2$ and $n\geq 2k+2$, then the set of facets of $D(n,2k)$ is given by 
	\begin{enumerate}
		\item $\tau\cup \{ i_k, i_k+1, i_k+3\}$, where $\tau\in B([1,i_k-1], 2k-3)$ and $i_k\leq n-3$,
		\item $\tau\cup \{ i_k, i_k+2, i_k+3\}$, where $\tau\in B([1,i_k-1], 2k-3)$ and $i_k\leq n-3$,
		\item $\tau\cup \{n-2, n-1, n\}$, where $\tau \in B([1, n-3], 2k-3)$,
		\item $\tau \cup \{1, i_k, i_k+1, i_k+2, i_k+3\}$, where $\tau\in B([2, i_k-1], 2k-5)$ and $i_k\leq n-3$.
	\end{enumerate}
\end{lemma}

The following result further emphasizes similarities between $D(n,d-1)$ and $\partial C(n,d)$. It is an immediate consequence of Lemma \ref{lem:facets-of-D} (along with Lemma \ref{lm: cyclic polytope properties}).

\begin{lemma}\label{lm: sibling properties} For all $k\geq 2$,
	\begin{eqnarray*}
	\lk(\{n-1, n\}, D(n, 2k-1))&=&\partial C(n-2,2k-2), \mbox{ and}\\
	\lk(\{n-2, n-1, n\}, D(n,2k))&=&\partial C(n-3,2k-2).
	\end{eqnarray*}
\end{lemma}

Note that for any $d$, $D(d+1,d-1)$ is a $(d-1)$-sphere with $d+1$ vertices; hence it is the boundary of the $d$-simplex. We are now ready to show that, similarly to the cyclic polytope, $D(n+1,d-1)$ is obtained from $D(n,d-1)$ by sewing.   

\begin{lemma} \label{lem:sewing}
	For $d\geq 4$ and $n> d$, define $$K(n,d-1)=\begin{cases}
		\overline{\{n-1,n\}}*B([1,n-2], 2k-3) & \text{ if }d=2k\\
		\overline{\{n-2, n-1, n\}}* B([1, n-3], 2k-3) & \text{ if }d=2k+1.		
	\end{cases}$$ Then $D(n+1, d-1)$ is the complex obtained from $D(n, d-1)$ by replacing $K(n,d-1)$ with $\partial K(n,d-1)*(n+1)$.
\end{lemma}
\begin{proof}
Assume $d=2k$ and let $J=(2,\ldots,2,3)$ where  $\|J\|=2k+1$. As we saw in the proof of Lemma \ref{lem:m-neighb-m-stacked}, 
\begin{equation} \label{eq:union-of-balls}
\Gamma^J_{n+1}=\Gamma^J_n\cup\left( \overline{\{n-1, n, n+1\}}*B([1,n-2], 2k-3)\right)=\Gamma^J_n\cup \left( K(n, 2k-1)*(n+1)\right)\end{equation} while
$\Gamma^J_n \cap \left(K(n, 2k-1)*(n+1) \right) = K(n, 2k-1)$.
	The claim follows by computing the boundary complexes of the balls on the left- and right-hand sides of \eqref{eq:union-of-balls}. The proof in the case of $d=2k+1$ is similar.
\end{proof}

\begin{lemma} \label{lem:polytopality}
	The complex $D(n,d-1)$ is the boundary complex of a simplicial polytope.
\end{lemma}
\begin{proof}
Let $k=\lfloor (d+1)/2\rfloor$, and let $F_1=\{n-1, n\}$, $F_2=\{n-2, n-1,n\}$, $\dots$, $F_{2k-1}=\{n-2k+1, \dots, n\}$, $F_{2k}=\{n-2k, \dots, n\}$. It follows  from Lemmas \ref{lm: cyclic polytope properties} and \ref{lem:sewing} 	
  that 
  $$ K(n, d-1)=\begin{cases}
  	\st(F_1)\backslash \left(\st(F_2)\backslash \big( \dots\backslash \big(\st(F_{2k-2})\backslash \st(F_{2k-1})\big)\dots\big)\right) & \mbox{if } d=2k\\
  	\st(F_2)\backslash \left(\st(F_3)\backslash \big( \dots\backslash \big(\st(F_{2k-1})\backslash \st(F_{2k})\big)\dots\big)\right) & \mbox{if }d=2k-1;
  \end{cases}$$
here the stars are computed in $D(n, d-1)$. 
	The polytopality of $D(n,d-1)$  is then a consequence of \cite{Shemer} and Lemma \ref{lem:sewing}: indeed, to obtain $D(n,d-1)$, we start with (the boundary of) the $d$-simplex (which is a simplicial polytope), and at each step we are sewing a new vertex whose position in $\R^d$ is determined by a flag; that is, we are performing an operation that, according to \cite[Lemma 4.4]{Shemer}  preserves polytopality.
\end{proof}

Returning to the topic of transversal numbers, we note that as follows from Lemma \ref{lem:facets-of-D}, the sets $\{1,3,5,7, \ldots\}\cap[n]$ and $\big(\{1,2,6,7,11,12,\ldots\}\cap [n]\big)\cup \{n\}$ are transversals of $D(n,2k-1)$ and $D(n,2k)$, respectively, and so $\tau(D(n,2k-1))\leq 1/2$ while $\tau(D(n,2k))\leq 2/5$. In fact, the description of facets from Lemma \ref{lem:facets-of-D} implies the following result.

\begin{lemma} \label{lem:transversal-of-D}
	For $k\geq 2$, $T(D(n, 2k-1))=\frac{n}{2}-O(1)$ and $T(D(n, 2k))=\frac{2n}{5}-O(1)$. 
\end{lemma}

Since the proof of Lemma \ref{lem:transversal-of-D} is very similar to that of \cite[Proposition 6.11]{NZ-transversal} (see also the proof of Lemma \ref{lem: (k+5)/2(k+6)} below), we only sketch the proof, and we only consider the case of $D(n,2k)$.

{\smallskip\noindent {\it Proof (Sketch): \ }
Consider a transversal $\T$ of $D(n,2k)$. We are interested in the sizes of maximal w.r.t.~inclusion intervals $[i,j]=\{i,i+1,\ldots,j\}$ contained in $\T$ and $\T^c=[n]\backslash \T$, respectively. To avoid any possible confusion, we note that the size of $[i,j]$ is $j-i+1$. The facets of $D(n,2k)$ are described in Lemma \ref{lem:facets-of-D}. Using the facets of type 1 and 2, one proves the following claims:
\begin{enumerate}
\item The union of all intervals of size $\geq 4$ in $\T^c$ has cardinality $O(k)$.
\item The number of singleton intervals of $\T$ that are followed by a non-singleton interval of $\T^c$ is $\leq k$.
\end{enumerate}
These claims then imply that  there exists a set $C\subset [n]$ of size $O(k)$ such that $[n]\backslash C$ is the disjoint union of pairs of adjacent intervals $(I,J)$, where $I\subset \T$, $J\subset \T^c$, and each pair $(I,J)$ satisfies (a) $|I|=2, |J|= 3$ or (b) $|I|\geq  |J|\geq 1$ or (c) $J=\emptyset$. This yields the desired bound $T(D(n, 2k))=\frac{2n}{5}-O(k)$.
\hfill$\square$\medskip}

 Since by Lemma \ref{lem:polytopality}, $D(n, d-1)$ is the boundary complex of a polytope, Lemma \ref{lem:transversal-of-D} yields the promised lower bound on $\tau_{2k+1}^P$:
\begin{theorem}  \label{thm:poly-2/5}
	For $k\geq 2$, $\tau^P_{2k}\geq \frac{1}{2}$ and $\tau^P_{2k+1}\geq \frac{2}{5}$.
\end{theorem}

We close this section with several remarks.

\begin{remark} Comparing the transversal numbers of $D(n, 2k)$ and $\partial C(n, 2k+1)$, which are $\frac{2n}{5}-O(1)$ and $2$, respectively, we conclude that for $n$ sufficiently large, these spheres are not combinatorially equivalent. Similarly, for $n$ sufficiently large, $D(n, 2k-1)$ and $\partial C(n, 2k)$ are not combinatorially equivalent either. This can be seen by observing that the links of $(2k-3)$-faces of $\partial C(n, 2k)$ are either $3$-, $4$-, or $(n-2k+2)$-cycles. On the other hand, the link of any $F=\{i_2,i_2+1,\dots, i_{k-1}, i_{k-1}+1, i_k, i_k+2\} \subset[2,n-2]$ in $D(n, 2k-1)$ is a cycle that contains all the vertices in $[1, i_k+4]$ except for those in $F$ and $i_k+3$. 
(In fact, as was checked with Sage, $D(8,3)$ and $\partial C(8, 4)$ are already non-isomorphic.)
\end{remark} 

\begin{remark} \label{rm:D(n,d)-vs-positive}
 Our definition of $D(n,d)$ was originally conceived by looking at a certain family
$\Delta^d_n$ of centrally symmetric (cs, for short) $\lceil d/2 \rceil$-neighborly $d$-spheres \cite{NZ-cs-neighborly, NZ-cs-neighborly-new}. The vertex set of the complex $\Delta^d_n$ is $\{\pm 1,\ldots,\pm n\}$, and we refer to a face of $\Delta^d_n$ as \emph{positive} if it is a subset of $[n]$. In the case of $d=3$ and $4$, the complete list of positive facets of $\Delta^d_n$ is known; see \cite[Lemma 3.1]{NZ-cs-neighborly-new} and \cite[Theorem 5.13]{Pfeifle-20}. Specifically, the list of positive facets of $\Delta^3_n$ consists of
\[ \{i, i+1,\ell,\ell+2\} \mbox{ (for  $1\leq i < i+1<  \ell \leq n-2$)}  \mbox{ and } \{i,i+1,n-1,n\} \mbox{ (for  $1\leq i  \leq  n-3$)}.\]
These facets form a $2$-neighborly $2$-stacked $3$-ball that is a subcomplex of $D(n,3)$. The remaining facets of $D(n,3)$ are  $\{1, \ell, \ell+1, \ell+2\}$ for $2\leq \ell\leq n-2$, and they form a $1$-stacked $3$-ball. Similarly, the list of positive facets of $\Delta^4_n$ consists of the following subsets of $[n]$: $$\{i <i+1<\ell< \ell+1< \ell+3\}, \{i< i+1< \ell< \ell+2< \ell+3\}, \{i< i+1< n-2< n-1< n\}.$$
These facets form a subcomplex of $D(n, 4)$; the remaining facets are $\{1, \ell, \ell+1, \ell+2, \ell+3\}$ for $2\leq \ell \leq n-3$, and they again form a $1$-stacked $4$-ball. In other words, for $d=3,4$, the subcomplex of the cs-$\lceil d/2\rceil$-neighborly $d$-sphere $\Delta^d_n$ generated by the positive facets can be completed to a non-cs $\lceil d/2 \rceil$-neighborly $d$-sphere $D(n, d)$.
\end{remark}

\section{Transversal ratios of spheres}
The goal of this section is to construct spheres of dimensions $3$, $4$, and $5$ with larger transversal ratios than the current record. One approach to do so is as follows. Start with a PL $(d-1)$-sphere $\Delta$ with a relatively large transversal ratio (such as $\partial C(n, 2k)$ or $D(n, 2k)$). Then apply a sequence of bistellar flips, or local retriangulations, to ensure that the resulting complex (another PL sphere, by Pachner's theorem) has an even larger transversal ratio. We flesh out the details of this approach in Section 5.1. Then in Sections 5.2 and 5.3 we provide specific constructions.

\subsection{Retriangulations}

Let $\Delta$ be a PL $(d-1)$-sphere. A local retriangulation of $\Delta$ is defined as follows. Consider a collection of PL $(d-1)$-balls $\mathcal{B}$ in $\Delta$ (each with a small number of vertices) that are pairwise vertex-disjoint. For each $B\in\mathcal{B}$, find a new PL $(d-1)$-ball $B'$ such that $\partial B'=\partial B$; we call such $B'$ a {\em retriangulation} of $B$. Let $\Delta'$ be obtained from $\Delta$ by replacing each $B\in \mathcal{B}$ with the corresponding $B'$. If $\Delta'$ is a PL $(d-1)$-sphere, then we say that $\Delta'$ is a {\em local retriangulation} of $\Delta$. (For example, any complex obtained from $\Delta$ by a bistellar flip is a local retriangulation of $\Delta$.) In order to guarantee that $\Delta'$ is a PL sphere, some mild restrictions on the balls in $\mathcal{B}$ are needed. One set of such restrictions is given by the following simple lemma whose proof we omit. 

\begin{lemma}\label{lem: retriangulation of PL spheres}
	Let $\Delta$ be a PL $(d-1)$-sphere, let $B\subset \Delta$ be a PL $(d-1)$-ball, and let $K$  be a PL $(d-1)$-ball such that $\partial K=\partial B$ and, in addition, $V(K)=V(B)$ or $V(K)\cap V(\Delta)=V(\partial K)$. If $B$ is an induced subcomplex of $\Delta$, then $K\cap (\Delta\backslash B)=\partial K$. In particular, replacing $B$ with $K$ in $\Delta$ results in a PL $(d-1)$-sphere.		
\end{lemma}

In the case of $d=2k$, let $\Delta=\partial C(n, 2k)$ and consider a family of disjoint PL $(d-1)$-balls $\mathcal{ B}=\{B([a_i, b_i], 2k-1): [a_i, b_i] \subsetneq [1,n], b_i-a_i\geq 2k\}$ in $\Delta$. Since (as follows from the Gale evenness condition) each $B([a_i,b_i],2k-1)$ is an induced subcomplex of $\partial C(n,2k)$, Lemma \ref{lem: retriangulation of PL spheres} applies. It remains to find appropriate retriangulations of $B([a_i, b_i], 2k-1)$ that increase the transversal ratio of the resulting complex. (Recall that $\tau(\partial C(n, 2k))=\frac{1}{2}-o(1)$.) Some sufficient conditions guaranteeing such an increase are given in the following lemma and remark. To simplify notation, for $F=\{i_1, \dots, i_k\}\subset [n]$, we let $F+j=\{i_1+j, \dots, i_k+j\}$. We also define $\Delta+i=\{\tau+i: \tau\in \Delta\}$.

\begin{lemma} \label{rem:Gamma_n}
 Assume $\Delta_n$ is the pure $(2k-1)$-dimensional complex  on vertex set $[(4k-1)n]$ whose set of facets consists of 
\begin{enumerate} 
\item all facets of $B([1, (4k-1)n], 2k-1)$ except for the facets of the following $n$ balls: $B([1, 4k-1], 2k-1)+(4k-1)m$, where $0\leq m\leq n-1$; 
\item the sets  $\{1,3,\dots, 4k-1\}+(4k-1)m$, where $0\leq m\leq n-1$. 
\end{enumerate}
Then  $\lim_{n\to\infty}\tau(\Delta_n)=\frac{2k}{4k-1}$.
\end{lemma}
\proof
Consider the set $$\bigcup_{m=0}^{n-1}\{1,3, \dots, 4k-1\}+(4k-1)m.$$
It forms a transversal of  $\Delta_n$, and hence  $\tau(\Delta_n)\leq \frac{2k}{4k-1}$.

Now, let $\T_n$ be any transversal of $\Delta_n$. For $0 \leq m\leq n-1$, consider the interval $I_m:=\{1, 2, \dots, 4k-1\}+(4k-1)m$. Call  the set  $I_m\backslash \T_n$ {\em bad} if it contains a pair of consecutive elements. Note that at most $k-1$ of the sets $I_0\backslash \T_n, \dots, I_{n-1}\backslash \T_n$   could be bad: if there were $k$ bad sets, then the union of $k$ pairs of consecutive elements, one from each bad set, would form a facet of type 1 that is disjoint from $\T_n$.
 In addition, facets of type 2 guarantee that $I_m\cap \T_n$ must contain at least one odd integer if $m$ is even, and at least one even integer  if $m$ is odd. This implies that at least $n-(k-1)$ of the sets $I_m\cap \T_n$ have size $\geq 2k$.
Thus
$$\tau(\Delta_n)=\frac{|\T_n|}{(4k-1)n}\geq \frac{2k}{4k-1}\left(1-\frac{k-1}{n}\right)=\frac{2k}{4k-1}-o(1).$$ The lemma follows.
\endproof

\begin{remark} \label{lem:4kn}
	Let $\Delta'_n$ be the pure $(2k-1)$-dimensional complex on vertex set $[4kn]$ whose set of facets is given by  (1) all facets of $B([1, 4kn], 2k-1)$ except for the facets of the balls $B([1, 4k], 2k-1)+4km$, where $0\leq m \leq n-1$, and (2)
	 the sets $\{1, 3, 5,\dots, 4k-1\}+4km$ and $\{2, 4, 6,\dots, 4k\}+4km$, where $0\leq m\leq n-1$.
    Then a proof similar to the one above shows that $\lim_{n\to\infty}\tau(\Delta_n)=\frac{4k+1}{8k}$.
\end{remark}

We are interested in whether a complex $\Delta$ as in Lemma \ref{rem:Gamma_n} or in Remark \ref{lem:4kn} can be completed to a simplicial sphere (or even the boundary of a simplicial polytope) of the same dimension and on the same vertex set  as $\Delta$. In view of Lemma \ref{lem: retriangulation of PL spheres} and part 4 of Lemma \ref{lm: cyclic polytope properties}, this leads to
\begin{question}\label{Question}
	Let $k\geq 2$.  
	\begin{enumerate}
		\item Is there a triangulation of $C(4k-1, 2k-1)$ that contains the facet $\{1,3,\dots, 4k-1\}$? \item Is there a triangulation of $C(4k, 2k-1)$ that contains the facets $\{1,3,\dots, 4k-1\}$ and $\{2,4,\dots, 4k\}$? 
	\end{enumerate}
\end{question}

\subsection{Dimensions $3$ and $5$}
When $k=2$, the answer to both parts of Question \ref{Question} is yes.

For part 1, consider the complex $L_7$ generated by the facets 
$$\{1,2,3,7\}, \{1,3,4,5\}, \{1,3,5,7\}, \{3,4,5,7\}, \{1,5,6,7\}.$$
It is a simplicial $3$-ball with vertex set $[7]$ whose boundary complex coincides with that of $\partial B([1,7], 3)=\partial C(7,3)$. Thus $L_7$ is a desired triangulation of $C(7,3)$, and $L_7+i$ is a retriangulation of $B([1+i,7+i],3)=B([1,7],3)+i$ for all $i$.
{Lemmas \ref{lem: retriangulation of PL spheres} and \ref{rem:Gamma_n}, along with the observation that $\{1,3,5,7\}$ is a transversal of $L_7$,  then imply

\begin{theorem}  \label{thm:3-sphere-5/8}
Let $\Lambda_n$ be the complex obtained from $\partial C(n,4)$ by replacing each of the balls $B([1,7],3)+7m$, where $0\leq m\leq n/7-1$, with $L_7+7m$. Then $\Lambda_n$ is a PL $3$-sphere and $\lim_{n\to \infty} \tau(\Lambda_n)=4/7$. In particular, $\tau^S_{4}\geq \frac{4}{7}$.
\end{theorem}

For part 2, consider the complex $L_8$ generated by the following facets (in a shelling order):
$$\{2,5,3,7\},\{2,5,7,6\},\{2,5,6,4\},\{2,5,4,3\},$$
$$ \{1,2,3,7\}, \{2,6,7,8\}, \{1,2,7,8\},\{1,3,5,7\},\{2,4,6,8\},$$
$$\{1,5,6,7\},\{2,3,4,8\}, \{1,3,4,5\}, \{4,5,6,8\}.$$
It is a shellable ball with vertex set $[8]$. The boundary complex of $L_8$ is generated by the facets $\{1,2,8\}, \{1,7,8\}$,  and $\{1,i,i+1\}$ and $\{i,i+1,8\} $ for all $2\leq i\leq 6$. In other words, $\partial L_8$ coincides with that of $\partial B([1, 8],3)=\partial C(8,3)$. In particular, $L_8$ is a triangulation of $C(8,3)$, and hence for any $i$, $L_8+i$ is a retriangulation of $B([1+i, 8+i], 3)=B([1,8],3)+i$. 

The same argument (but using Lemma \ref{lem: retriangulation of PL spheres} and Remark \ref{lem:4kn}) implies that  if we start with $\partial C(n,4)$ and replace each $B([1,8],3)+8m$, where $0\leq m\leq n/8-1$ with $L_8+8m$, then the resulting complex is a PL $3$-sphere whose transversal ratio is  $9/16-o(1)$ as $n\to\infty$.

To close this section, we show that part 1 of Question \ref{Question} also has an affirmative answer when $k=3$. Below we provide a particular shelling order of one of the triangulations of $C(11,5)$ (this shelling order was generated by Sage):
$$\{1, 2, 3, 4, 5, 11\}, \{1, 2, 3, 5, 7, 11\}, \{2, 3, 5, 6, 7, 11\}, \{1, 2, 3, 5, 6, 7\}, \{1, 2, 3, 7, 9, 11\},$$
$$\{1, 2, 7, 8, 9, 11\}, \{1, 3, 4, 5, 10, 11\}, \{1, 2, 5, 6, 7, 11\}, \{1, 3, 5, 7, 9, 11\}, \{1, 3, 5, 9, 10, 11\},$$ 
$$\{1, 5, 7, 9, 10, 11\}, \{1, 5, 6, 7, 10, 11\}, \{1, 5, 6, 7, 9, 10\}, \{1, 7, 8, 9, 10, 11\}, \{3, 4, 5, 9, 10, 11\},$$ 
$$\{3, 4, 5, 7, 9, 11\}, \{5, 6, 7, 9, 10, 11\}, \{3, 4, 5, 6, 7, 11\}, \{1, 2, 3, 9, 10, 11\}, \{1, 3, 4, 5, 9, 10\},$$ 
$$\{1, 2, 3, 7, 8, 9\}, \{1, 3, 5, 7, 8, 9\},\{3, 4, 5, 7, 8, 9\}, \{1, 3, 4, 5, 8, 9\}, \{2, 3, 7, 8, 9, 11\},$$ 
$$\{1, 5, 6, 7, 8, 9\}, \{1, 3, 4, 5, 7, 8\},\{1, 3, 4, 5, 6, 7\}, \{3, 4, 7, 8, 9, 11\}, \{4, 5, 7, 8, 9, 11\},$$ 
$$\{5, 6, 7, 8, 9, 11\}.$$

Denote the above complex by $L_{11}$. It is not hard to check that every ridge of $L_{11}$ is contained in at most two facets. Together with the fact that the above ordering is a shelling, we conclude that $L_{11}$ is a PL $5$-ball. Furthermore, one can check that $\partial L_{11}=\partial B([1,11], 5)=\partial C(11,5)$.  Since $L_{11}$ contains $\{1, 3, 5, 7, 9, 11\}$ as a facet, Lemma \ref{lem: retriangulation of PL spheres} and Lemma \ref{rem:Gamma_n}, along with the observation that $\{2,4,6,8,10,11\}$ is a transversal of $L_{11}$, yield

\begin{theorem} \label{thm:5-sphere-6/11}
Let $\Pi_n$ be the complex obtained from $\partial C(n,6)$ by replacing each of the balls $B([1,11],5)+11m$, where $0\leq m\leq n/11-1$, with $L_{11}+11m$. Then $\Pi_n$ is a PL $5$-sphere and $\lim_{n\to\infty}\tau(\Pi_n)=6/11$.
In particular, $\tau^S_{6}\geq \frac{6}{11}$.
\end{theorem}

\subsection{Dimension $4$}

The goal of this section is to prove the following result:
\begin{theorem} \label{thm:1/2-in-dim4}
	$\tau^S_5\geq \frac{1}{2}$.
\end{theorem}

Our strategy is to apply sequences of bistellar flips to the family of PL spheres, $D(n, 4)$, introduced in Section 4. This will require the following lemma and definition.
\begin{definition} \label{def:Gamma_n,k}
	Let $\Gamma_{n,0}=D(n,4)$. Assume that for a fixed $k\geq 0$ and for all $k+2\leq i\leq n-5$, the set $\{i-k-1, i+1, i+5\}$ is not a face of $\Gamma_{n, k}$ while the set $\{i-k, i+2, i+4\}$ is a face of $\Gamma_{n, k}$, and furthermore, that $$\st(\{i-k, i+2, i+4\}, \Gamma_{n, k})=\overline{\{i-k, i+2, i+4\}}* \partial \overline{\{i-k-1, i+1, i+5\}}.$$ Note that no two such stars share a common facet. 
	Hence we can apply bistellar flips simultaneously to all stars $\st(\{i-k, i+2, i+4\}, \Gamma_{n, k})$ for $k+2\leq i\leq n-5$. We define $\Gamma_{n,k+1}$ to be the resulting complex. 
\end{definition}
This definition is justified by the following lemma.

\begin{lemma}\label{lm: gamma_nk}
For all $n\geq 6$ and all $0\leq k\leq n-6$, the complex $\Gamma_{n,k}$ satisfies the following properties.
	\begin{enumerate} 
	\item Each of the following sets is a facet of $\Gamma_{n,k}$:
	\begin{enumerate}
		\item $\{i,i+1,j,j+1,j+3\}$, where $i\geq 1$ and  $i+k+1<j\leq n-3$;
		\item $\{i,i+1,j,j+2,j+3\}$, where $i\geq 1$ and $i+k+2<j\leq n-3$;
		\item $\{i-k,i+1,i+2,i+4,i+5\}$, where $k+1\leq i\leq n-5$;
		\item $\{i-\ell,i-\ell+1,i+1,i+4,i+5\}$, where $1\leq \ell\leq k$ and $\ell<i\leq n-5$.
	\end{enumerate}
		\item  For all $k+2\leq i\leq n-5$, $\{i-k, i+2, i+4\}$ is a face of $\Gamma_{n, k}$, and
		$$\st(\{i-k, i+2, i+4\}, \Gamma_{n, k})=\overline{\{i-k, i+2, i+4\}}* \partial \overline{\{i-k-1, i+1, i+5\}}.$$
		\item For all $k+2\leq i\leq n-5$, $\{i-k-1, i+1, i+5\}$ is not a face of $\Gamma_{n, k}$.		
	\end{enumerate}
\end{lemma}
\noindent Note that the collection of sets in part 1(d) of this lemma is empty when $k=0$.

\begin{proof}
	The proof is by induction on $k$. That $\Gamma_{n,0}=D(n,4)$ satisfies these properties follows from Lemma \ref{lem:facets-of-D}.
	
	Assume that the statement holds for some $k\leq m$. To verify that the properties continue to hold for $\Gamma_{n,m+1}$, note that when applying bistellar flips to $\Gamma_{n,m}$, we replace the facets
	$$\{i-m-1, i-m, i+1, i+2, i+4\}, \{i-m-1, i-m, i+2, i+4, i+5\}, \{i-m, i+1, i+2, i+4, i+5\}$$ 
	with the facets
	$$\{i-m-1, i-m, i+1, i+2, i+5\}, \{i-m-1, i-m, i+1, i+4, i+5\}, \{i-m-1, i+1, i+2, i+4, i+5\}$$
	for all $m+2\leq i\leq n-5$. All other facets of $\Gamma_{n,m}$ remain facets of $\Gamma_{n, m+1}$. This implies that the first property continues to hold. 
	
	Since, in addition to $\{i-m-1, i+1, i+2, i+4, i+5\}$, the complex $\Gamma_{n,m+1}$ has $\{i-m-2, i-m-1, i+1, i+2, i+4\}$ and $\{i-m-2, i-m-1, i+2, i+4, i+5\}$ as its facets, we conclude that the second property also holds, that is, $$\st(\{i-m-1, i+2, i+4\},\Gamma_{n, m+1})=\overline{\{i-m-1, i+2, i+4\}}*\partial \overline{\{i-m-2,i+1, i+5\}}.$$
	
	Finally, to see that the third property holds, note that by Lemma \ref{lem:facets-of-D}, $\{i-m-2, i+1, i+5\}\notin \Gamma_{n,0}=D(n,4)$. The bistellar flips replace all $2$-faces of the form $\{i-\ell, i+2, i+4\}$ with $\{i-\ell-1,i+1, i+5\}$, for all $\ell$ and $i$ such that $0\leq \ell\leq m$ and $\ell+2\leq i\leq n-5$, but do not introduce any other $2$-faces. Hence $\{i-m-2, i+1, i+5\}\notin \Gamma_{n,m+1}$.
\end{proof}

The significance of the complexes $\Gamma_{n,k}$ is explained by the following lemma. The beginning of the proof of this lemma is similar to that of \cite[Proposition 6.11]{NZ-transversal}.

\begin{lemma}  \label{lem: (k+5)/2(k+6)}
	Let $k\geq 1$. Then $\liminf_{n\to\infty} \tau(\Gamma_{n, k})\geq \frac{k+4}{2(k+5)}.$
\end{lemma}
\begin{proof} Consider a transversal $\T$ of $\Gamma_{n, k}$. We are interested in the sizes of maximal w.r.t.~inclusion intervals $[i,j]=\{i,i+1,\ldots,j\}$ contained in $\T^c$ and $\T$, respectively. We start by establishing the following claims. (To avoid any possible confusion, we note that the size of $[i,j]$ is $j-i+1$.)

\smallskip\noindent{\bf Claim 1:} No interval in $\T^c$ can have size $\geq 7$.

If $[j, j+6]\subseteq \T^c$ for some $j$, then $\{j, j+1, j+2, j+5, j+6\}$ is a facet of $\Gamma_{n, k}$ (see Lemma \ref{lm: gamma_nk} part 1(d)). However, this facet is disjoint from $\T$, contradicting  our assumption that $\T$ is a transversal. \hfill$\square$\medskip

\smallskip\noindent{\bf Claim 2}: There exists a subset $[a,b]\subset [n]$ with $b-a\leq k+4$ that contains {\bf all} intervals of $\T^c$ of size $\geq 4$.

If $\T^c$ contains at most one interval of size $\geq 4$, then since $k\geq 1$, the result follows from Claim 1. Otherwise, let $[i,i+p]$ and $[j,j+q]$ be the left-most and the right-most such intervals (and so $p,q\geq 3$). We must have $j+q-i\leq k+4$, or else $\{i,i+1,j+q-3,j+q-2,j+q\}$ would satisfy $(j+q-3)-i>k+1$, and hence (by Lemma \ref{lm: gamma_nk} part 1(a)) it would be a facet of $\Gamma_{n,k}$ that is disjoint from $\T$. The claim follows. \hfill$\square$\medskip

\smallskip\noindent{\bf Claim 3:} Let $\{i_1\},\dots,\{i_p\}$ be the list of all maximal intervals in $\T$ of size $1$, each of which is followed by an interval in $\T^c$ of size $\geq 2$. Then $i_p-i_1\leq k+4$.

If  $\{i_1\},\dots,\{i_p\}$ are such intervals with $i_1<\dots<i_p$ and $i_p-i_1>k+4$, then $F:=\{i_1+1, i_1+2, i_p-1, i_p+1,i_p+2\}$ is contained in $\T^c$. However, $(i_p-1)-(i_1+1)=(i_p-i_1)-2 >k+2,$ and so (by Lemma \ref{lm: gamma_nk} part 1(b)) $F$ is a facet of $\Gamma_{n,k}$ that is disjoint from $\T$,  contradicting our assumption that $\T$ is a transversal.  \hfill$\square$\medskip

Let $C\subset [n]$ be the union of the following sets:
\begin{itemize}
\item the smallest interval $[a,b]$ from Claim 2 that contains all intervals of $\T^c$ of size $\geq 4$;
\item the smallest interval that contains all $\{i_1\},\dots,\{i_p\}$  from Claim 3, together with the interval from $\T^c$ that trails $\{i_p\}$;
\item if $1\in\T^c$, then also the left-most interval of $\T^c$.
\end{itemize}
The above three claims guarantee that $|C|=O(k)$, and by the definition of $C$, $[n]\backslash C$ is the union of at most three intervals. Each such interval $[x,y]$
can be written as the disjoint union of pairs of adjacent intervals $(I, J)$ such that $I \subseteq \T$, $J \subseteq \T^c$,  and each pair $(I, J)$ satisfies (a) $|I|=2$, $|J|= 3$ or (b) $|I|\geq |J|\geq 1$. (We also allow $J=\emptyset$ for the last pair $(I,J)$ of $[x,y]$). This already implies that $|\T|=2n/5-O(k)$, which is a weaker bound than the one promised in the statement. To improve this bound we make use of the facets of $\Gamma_{n,k}$ described in Lemma \ref{lm: gamma_nk} part 1(d).

 For convenience, we say that a pair $(I,J)$ is {\em of type} $(|I|,|J|)$. As above, consider the interval $[x,y]$ (which is one of at most three intervals comprising $[n]\backslash C$).
Assume that $I=\{i, i+1\}$ and $J=\{i+2,i+3,i+4\}$ form a pair $(I,J)$ of type $(2,3)$ in $[x,y]$, and let $I'=\{j, j+1\}$ and $J'=\{j+2,j+3,j+4\}$ be the nearest pair of type $(2,3)$ to the right of $(I,J)$ in $[x,y]$. (The case where $(I,J)$ is the right-most pair of $[x,y]$ of type $(2,3)$, and so no such $(I',J')$ exists, is discussed at the end of the proof.) Here $I, I'$ are subsets of $\T$, while $J$ and $J'$ are subsets of $\T^c$, and $j-1\in \T^c$ with $j-1\geq i+4$. Consider the set $F:=\{i+2,i+3,j-1,j+2,j+3\}$. This set is disjoint from $\T$, so it is not a facet, which  by Lemma \ref{lm: gamma_nk} part 1(d) means that we must have $(j-1)-(i+3)\geq k+1$. 

Consequently, the interval $[i+5, j-1]$ has size $X\geq k$ and, by our assumption that there are no pairs of type $(2,3)$ in this interval, it follows that $[i+5, j-1]$ is a disjoint union of pairs as in (b), i.e., pairs $(\tilde{I},\tilde{J})$ with $|\tilde{I}| \geq|\tilde{J}|\geq 1.$
In particular, at least half of the elements of the interval  $[i+5,j-1]$ are in $\T$. Therefore, 
$$\frac{\left|\T\cap [i, j-1]\right|}{\left|[i,j-1]\right|} \geq \frac{2+ X/2}{5+X}=\frac{X+4}{2(X+5)} =\frac{1}{2}\left(1-\frac{1}{X+5}\right) \geq  \frac{1}{2} \left(1-\frac{1}{k+5}\right).$$  This argument applies to every pair $(I,J)$ of type $(2,3)$ in $[x,y]$ but the right-most one. Furthermore, all pairs $(\tilde{I},\tilde{J})$ in $[x,y]$ to the left of the left-most pair of type $(2,3)$ satisfy  $|\tilde{I}| \geq |\tilde{J}|\geq 1$ and so do all pairs to the right of the right-most pair of type $(2,3)$. This shows that  $|\T\cap[x,y]|\geq \frac{1}{2}\left(1-\frac{1}{k+5}\right)|[x,y]|-O(1)$ for each of at most three disjoint intervals comprising $[n]\backslash C$. Since $|[n]\backslash C|=n-O(k)$, 
we conclude that
$$
|\T|\geq  \frac{1}{2}\left(1-\frac{1}{k+5}\right)n-O(k).
$$
The statement follows.
\end{proof}

Since the complexes $\Gamma_{n,k}$ are obtained from $D(n,4)$ by a sequence of bistellar flips, all these complexes are PL $4$-spheres. Thus Theorem \ref{thm:1/2-in-dim4} follows by applying Lemma \ref{lem: (k+5)/2(k+6)} for values of $k\to\infty$. In fact, one can easily check that for all $n$ and $k$, $\{1,3,5,\dots\}\cap[n]$ is a transversal of $\Gamma_{n,k}$, and so using these spheres, we cannot beat the  $1/2$ bound.

\begin{remark}
The facets of $D(n,4)$ listed in 
Lemma \ref{lm: gamma_nk} (for $k=0$) are also the positive facets of the cs neighborly $4$-sphere $\Delta^4_n$; see \cite{NZ-cs-neighborly, NZ-cs-neighborly-new} and Remark \ref{rm:D(n,d)-vs-positive}. Hence, similar to the case of $\Gamma_{n, k}$, one can start with $\Delta^4_n$ and apply bistellar flips described in Definition \ref{def:Gamma_n,k} to the stars of certain positive faces as well as to the stars of their antipodes to produce a sequence of cs $4$-spheres whose transversal ratios converge to $1/2$ as $n\to\infty$.
\end{remark}

\section{Open problems}
We conclude with several open problems. We think of $\tau^P_d$ and $\tau^S_d$ as $[0,1]$-valued functions of $d$. Our current knowledge of the values of these functions can be summarized as follows:
 $$\tau^P_2=\tau^P_3=1/2, \mbox{ and }\tau^P_{2k}\geq 1/2,\; \tau^P_{2k+1}\geq 2/5, \mbox{ for } k\geq 2;$$
$$\tau^S_2=\tau^S_3=1/2, \; \tau^S_4\geq 4/7, \;\tau^S_5\geq 1/2, \; \tau^S_6\geq 6/11,  \mbox{ and }\tau^S_{2k}\geq 1/2,\; \tau^S_{2k-1}\geq 2/5, \mbox{ for } k\geq 4.$$

It would be most desirable to determine explicit values of $\tau^P_d$ and $\tau^S_d$ for some $d\geq 4$. While this is completely out of reach at the moment, the following problems might be more accessible. To start, observe that by Proposition \ref{prop: pure}, for any $n^{(d+1)/2} \gg m \gg n$, there exists a  sequence of pure simplicial complexes of dimension $d-1$ with $n$ vertices and $\Theta(m)$ facets whose transversal ratios converge to $1$ as $n\to\infty$. In stark contrast, despite the fact that the numbers of facets of simplicial $(d-1)$-spheres with $n$ vertices range from $\Theta(n)$ to $\Theta(n^{\lfloor d/2\rfloor})$ (see \cite{Barnette-LBT-pseudomanifolds, Barnette73,McMullen70, Stanley75}), when $d>6$, we do not have a single example of a simplicial $(d-1)$-sphere whose transversal ratio is larger than $1/2$; yet for $d\geq 4$, we also do not have any non-trivial upper bounds on $\tau^P_d$ and $\tau^S_d$. This leads to

\begin{problem}
	For a fixed $d\geq 4$, are $\tau^P_d$ and $\tau^S_d$ bounded away from 1?
\end{problem}
A variation of the above problem can be found in \cite[Section 10.3]{Alonetall02}. The next problem concerns the asymptotic behavior of $\{\tau^P_d\}$ and $\{\tau^S_d\}$.
\begin{problem}
	Is it true that $\lim_{d\to \infty} \tau^P_d= \lim_{d\to \infty} \tau^S_d=1$? 	Are the sequences $\{\tau^P_d\}$ and $\{\tau^S_d\}$ (weakly) increasing? Or at least, are the sequences $\{\tau^P_{2k}\}$, $\{\tau^P_{2k+1}\}$, $\{\tau^S_{2k}\}$, and $\{\tau^S_{2k+1}\}$ (weakly) increasing?
\end{problem}

It is also worth mentioning that, at the moment, we do not know  whether the infinite families of spheres of dimension 3, 4, and 5 constructed in Section 5 are polytopal or not. (We expect they are not polytopal.) Consequently, there are gaps between the existing lower bounds on $\tau^P_d$ and $\tau^S_d$ for $4\leq d\leq 6$. This prompts us to ask
\begin{problem}
	Is $\tau^P_d=\tau^S_d$ for all $d\geq 4$?
\end{problem}

Finally, making any additional progress on Tur\'an's problem, especially for such classes of pure simplicial complexes as pseudomanifolds or Eulerian complexes, would be of great importance.

\section*{Acknowledgments}  We are grateful to the anonymous referees for pointing out mistakes in earlier versions of Theorem \ref{thm:3-sphere-5/8}  and Lemma \ref{lem: (k+5)/2(k+6)} 
as well as for many other helpful suggestions that improved the presentation.
	{\small
	\bibliography{refs}
	\bibliographystyle{plain}
	}
\end{document}